\documentclass[12pt]{amsart}
\usepackage{graphicx}
\usepackage{hyperref}
\usepackage{setspace}
\usepackage{enumerate}
\usepackage{amsmath,amsfonts,amssymb,amscd}
\usepackage{bbm}
\usepackage{enumitem}
\usepackage{amsthm}
\newtheorem{thm}{Theorem}[section]

\newtheorem{lem}[thm]{Lemma}

\newtheorem{example}[thm]{Example}
\theoremstyle{definition}

\theoremstyle{remark}

\numberwithin{equation}{section}

\let\emptyset\varnothing

\newcommand{\FF}{\mathbb{F}}
\newcommand{\ZZ}{\mathbb{Z}}

\newcommand{\QQ}{\mathbb{Q}}
\newcommand{\RR}{\mathbb{R}}
\newcommand{\CC}{\mathbb{C}}

\newcommand{\y}{\mathbf{y}}

\renewcommand{\t}{\mathbf{t}}
\renewcommand{\a}{\mathbf{a}}

\renewcommand{\c}{\mathbf{c}}
\newcommand{\x}{\mathbf{x}}

\newcommand{\calI}{\mathcal{I}}

\newcommand{\calX}{\mathcal{X}}

\newcommand{\tet}{{\theta}}
\newcommand{\bfalp}{{\boldsymbol{\alpha}}}

\newcommand{\bfa}{\mathbf{a}}

\newcommand{\bfe}{\mathbf{e}}

\newcommand{\bfx}{\mathbf{x}}

\newcommand{\bal}{\boldsymbol{\alpha}}

\newcommand{\tsp}{\thinspace}
\begin{document}

\title[]{Birch's theorem on forms in many variables \\  with a Hessian condition}

\author{Shuntaro Yamagishi}
\address{IST Austria, Am Campus 1, 3400 Klosterneuburg, Austria}
\email{shuntaro.yamagishi@ist.ac.at}
\indent

\date{Revised on \today}

\begin{abstract}
Let $F \in \ZZ[x_1, \ldots, x_n]$ be a homogeneous form of degree $d \geq 2$, and
$V_F^*$ the singular locus of the hypersurface $\{\x \in \mathbb{A}^n_{\CC}: F(\x) = 0 \}$.
A longstanding result of Birch states that there is a non-trivial integral solution to the equation
$F(x_1, \ldots, x_n) = 0$ provided $n > \dim V_F^* + (d-1) 2^d$ and there is a non-singular solution in $\RR$ and $\QQ_p$ for all primes $p$.
In this article, we give a different formulation of this result. More precisely, we replace
$\dim V_F^*$ with a quantity $\mathcal{H}_F$ defined in terms of the Hessian matrix of $F$. This quantity satisfies
$0 \leq \mathcal{H}_F \leq \dim V_F^*$; therefore, we improve on the aforementioned result of Birch if $\mathcal{H}_F < \dim V_F^*$.
We also prove the corresponding result for systems of forms of equal degree.
%We also show that there are examples where $\mathcal{H}_F = 0$ while $\dim V^* = n/2$.
\end{abstract}



\maketitle

\section{Introduction}
Let $F_1, \ldots, F_R \in \ZZ[x_1, \ldots, x_n]$ be homogeneous forms of equal degree $d \geq 2$.
We define
$$
V_{\mathbf{F}}^* = \left\{ \x \in \mathbb{A}^n_{\CC}: \textnormal{rank }  \left[  \frac{ \partial F_{\ell} }{ \partial x_{i}} (\x)  \right]_{\substack{ 1 \leq \ell \leq R  \\  1 \leq i \leq n }}   < R   \right\}.
$$
A longstanding result of Birch \cite{Bir} states the following.
\begin{thm} [Birch]
\label{Birchthm}
Let $F_1, \ldots, F_R \in \ZZ[x_1, \ldots, x_n]$ be  homogeneous forms of degree $d \geq 2$.
Let $\mathcal{B} \subseteq [-1,1]^n$ be a box whose sides have lengths at most $1$ and are parallel to the coordinate axes.
Suppose
$$
n > \dim V_{\mathbf{F}}^* + R(R+1) (d-1) 2^d.
$$
Let $\delta > 0$ be sufficiently small. Then there exists $c_{\mathbf{F}} \geq 0$ such that
\begin{eqnarray}
\# \{ \x \in P \mathcal{B} \cap \ZZ^{n}: F_1(\x) = \cdots = F_R(\x)  = 0  \}
= c_{\mathbf{F}} P^{n - d} + O(P^{n - d - \delta}),
\end{eqnarray}
for $P > 0$. Furthermore, $c_{\mathbf{F}} > 0$ if there exists a non-singular solution to the system of equations
$F_1(\x) = \cdots = F_R(\x) = 0$ in the interior of $\mathcal{B}$ and in $\QQ_p$ for all primes $p$.
\end{thm}
This classical result of Birch has been greatly influential. For general information we refer to the survey \cite{TBsurvey} and for a comprehensive overview of the cubic case to the book \cite{TBbook}. We only mention some of the related results here. In the case $R = 1$, the number of variables required were improved in various work \cite{BP, H, MV}, while when $R > 1$ Rydin Myerson made a significant progress \cite{RM1, RM2, RM3}.
For systems of forms involving differing degrees, there are results by Schmidt \cite{S} and Browning and Heath-Brown \cite{BHB1}.
Birch's theorem over the function field $\FF_q(t)$ was established by Lee \cite{lee} and the techniques were further refined
with applications to algebraic geometry in the work of Browning and Vishe \cite{BV1, BV2} and Browning and Sawin \cite{BS}.
One can also consider Birch's theorem for integral points with prime coordinates, where Cook and Magyar \cite{CM} first achieved the major result
analogous to the above theorem in this setting, with further developments by the author \cite{Yam, Yam2} and Liu and Zhao \cite{LZ}.

In many of these problems involving techniques from Birch's theorem, it is customary to consider the codimension of $V_{\mathbf{F}}^*$, that
is the result holds if the codimension of $V_{\mathbf{F}}^*$ is greater than some quantity depending on the degree.
In this article, we prove a variant of Theorem \ref{Birchthm} in terms of a condition on the Hessian matrices of $F_1, \ldots, F_R$, in which
the number of variables required is always less than or equal to that in the above result of Birch.
It is also worth mentioning the work of Dietmann \cite{Di} and Schindler \cite{Sch}, where Birch's theorem is proved in terms of
the dimension of the singular loci of forms in the pencil.

Given a homogeneous form $G \in \ZZ[x_1, \ldots, x_n]$, we define $H_G$ to be the Hessian matrix of $G$, i.e.
$$
H_G(\x) = \left[  \frac{ \partial^2 G }{ \partial x_{i} \partial x_j} (\x) \right]_{\substack{ 1 \leq i \leq n  \\  1 \leq j \leq n }}.
$$
We note that if $\deg G = 2$, then every coordinate of $H_G (\x)$ is a constant; therefore, we write $H_G(\x) = H_G$ in this case.

Let us define
\begin{eqnarray}
\label{def}
\mathcal{H}_{G} =
\begin{cases}
\max_{0 \leq r \leq n}   \dim \{ \x \in \mathbb{A}^n_{\CC}:  \textnormal{rank }  H_{G}(\x)  \leq r  \} - r  & \mbox{if } d > 2,  \\
n - \textnormal{rank }  H_{G} & \mbox{if } d = 2.
\end{cases}
\end{eqnarray}
The quantity $\dim \{ \x \in \mathbb{A}^n_{\CC}:  \textnormal{rank }  H_{G}(\x)  \leq r  \}$ is not unusual and
has made appearances in the literature before, for example in \cite{BHB} and \cite{HB3} (though in the latter work, it is only for degree $3$ and in terms of the number of integral points). In fact, it is proved in \cite[Lemma 2]{BHB} that
\begin{eqnarray}
\label{Hineq}
0 \leq \mathcal{H}_{G} \leq \dim V_G^*.
\end{eqnarray}

In this article, we prove that we may replace $\dim V_{\mathbf{F}}^*$ in Theorem \ref{Birchthm} with
\begin{equation}\label{HESS}
\max_{\c \in \ZZ^R \setminus \{ \mathbf{0} \}} \mathcal{H}_{\c . \mathbf{F}},
\end{equation}
where $\c . \mathbf{F} = c_1 F_1 + \cdots + c_R F_R$.
\begin{thm}
\label{mainthm}
Let $F_1, \ldots, F_R \in \ZZ[x_1, \ldots, x_n]$ and $\mathcal{B}$ as in Theorem \ref{Birchthm}.
Suppose
$$
n > \max_{\c \in \ZZ^R \setminus \{ \mathbf{0} \}} \mathcal{H}_{\c . \mathbf{F}} +  R (R+1) (d-1) 2^d.
$$
Let $\delta > 0$ be sufficiently small. Then there exists $c_{\mathbf{F}} \geq 0$ such that
\begin{eqnarray}
\# \{ \x \in P \mathcal{B} \cap \ZZ^{n}: F_1(\x) = \cdots = F_R(\x)  = 0  \}
= c_{\mathbf{F}} P^{n - d} + O(P^{n - d - \delta}),
\end{eqnarray}
for $P > 0$. Furthermore, $c_{\mathbf{F}} > 0$ if the same local conditions as in Theorem \ref{Birchthm} are satisfied.
\end{thm}
It follows from (\ref{Hineq}) and \cite[pp. 209]{Sch} that
$$
0 \leq \mathcal{H}_{\c . \mathbf{F}} \leq \dim V_{\c . \mathbf{F}}^* \leq \dim V_{\mathbf{F}}^*
$$
for all $\c \in \ZZ^R \setminus \{ \mathbf{0} \}$.
Therefore, the number of variables required in Theorem \ref{mainthm} is never more than that in Theorem \ref{Birchthm}.
It seems likely that in many of the results related to Birch's theorem mentioned above, we may replace the dimension of $V_{\mathbf{F}}^*$
with (\ref{HESS}) as in our theorem. We shall present examples in Section \ref{EXAMP} where $\mathcal{H}_F = 0$ while $\dim V_F^* = n/2$.
Finally, it is interesting to note that in the author's work with Schindler \cite{SY}, in which an upper bound for the number of rational points with bounded height
on certain manifolds are established, a condition regarding linear combinations of Hessian matrices also appears in the statement of the result,
which may or may not have connections to the present work.

\textit{Acknowledgements.} The author was supported by the NWO Veni Grant \texttt{016.Veni.192.047} during his time at Utrecht University and
by the FWF grant P 36278 at the Institute of Science and Technology Austria while working on this article. He would also like to thank Tim Browning, Jakob Glas and Simon Rydin Myerson for useful suggestions and conversations.

\section{A variant of the classical Weyl differencing argument}
Given $G \in \ZZ[x_1, \ldots, x_n]$ a homogeneous form of degree $d \geq 2$, we write
\begin{eqnarray}
\label{def1}
G(\x) = \sum_{j_1, \ldots, j_d = 1}^n G_{\mathbf{j}} x_{j_1} \cdots x_{j_d},
\end{eqnarray}
where $G_{\mathbf{j}} \in \QQ$ are symmetric with respect to $\mathbf{j} = (j_1, \ldots, j_d)$.
Let us denote by $\Gamma_G$ the multilinear form associated to $G$, which is defined to be
\begin{eqnarray}
\Gamma_{G} (\x_1, \ldots, \x_d) = d! \sum_{\mathbf{j}} G_{\mathbf{j}}  x_{1, j_1} \cdots x_{d, j_d},
\end{eqnarray}
where the range of summation is the same  as in (\ref{def1}). We prove the following estimate in Section \ref{proof}.
\begin{lem}
\label{main lem}
Let $G \in \ZZ[x_1, \ldots, x_n]$ be a homogeneous form of degree $d \geq 2$. Let $B \geq 1$.
Then
\begin{eqnarray}
\notag
&&\# \{ \x_1, \ldots, \x_{d-1} \in ([-B, B] \cap \ZZ )^n:  \Gamma_G (\x_1, \ldots, \x_{d-1}, \mathbf{e}_i )  = 0  ~  (1 \leq  i \leq n)   \}
\\
\notag
&\ll& B^{(d-2)n + \mathcal{H}_G},
\end{eqnarray}
where the implicit constant depends only on $d$ and $n$.
\end{lem}
When $R=1$ the corresponding bound used  in \cite{Bir} for this quantity is $\ll B^{(d-2)n +  \dim V_F^*}$,
which results in the presence of $\dim V_F^*$ in the statement of Theorem \ref{Birchthm};
therefore, by using this estimate instead and following through the proof in \cite{Bir}, we obtain Theorem \ref{mainthm}.
However, we will follow the argument in \cite{Sch} which differs from that in \cite{Bir} when $R > 1$; this approach has the
benefit of being able to make use of Lemma \ref{main lem} easily for systems of forms.

%For completeness we now present some of the key steps to establish Theorem \ref{}.
%We will follow the exposition in \cite{SS} (with $P_1 = \cdots = P_n$ and without the smooth weight) of Birch's argument.
Let $F_1, \ldots, F_R \in \ZZ[x_1, \ldots, x_n]$ be homogeneous forms of degree $d$, and $\mathcal{B}$ as in Theorem \ref{Birchthm}.
For $P > 0$ and $\bal \in \RR^R$, we define
$$
S(\bal) = \sum_{\x \in P \mathcal{B} } \exp \left(2 \pi i  \sum_{\ell = 1}^R  \alpha_\ell F_\ell(\x) \right).
$$
For simplicity we denote $\Gamma_{\ell} = \Gamma_{F_\ell}$. Let $\varepsilon > 0$.
By following through the argument in \cite{Bir} or\footnote{This part of the argument in \cite{Bir}
is mostly referenced to \cite{D}.} \cite[pp. 13]{SS} (with $P = P_1 = \cdots = P_n$ and without smooth weights),
we obtain, for any $0 < \eta < 1$,
\begin{eqnarray}
\label{SSeqn}
|S(\boldsymbol{\alpha})|^{2^{d-1}}
\ll
P^{n( 2^{d-1} - d + 1) + \varepsilon } P^{ (1 - \eta ) (d-1) n  } \# N^{(d-1)}(P^{\eta}; \bal),
\end{eqnarray}
where
\begin{eqnarray}
N^{(d-1)}(P^{\eta}; \bal)
\notag
&=&
\{  (\x_1, \ldots, \x_{d-1}) \in ([- P^{\eta}, P^{\eta}] \cap \ZZ )^{n (d-1)}:
\\
\notag
&& \| \sum_{\ell = 1}^R \alpha_\ell \Gamma_{\ell} (\x_1, \ldots, \x_{d-1}, \mathbf{e}_i )  \| <
P^{ (\eta - 1) (d-1) - 1}  ~  (1 \leq  i \leq n)  \}
\end{eqnarray}
and the implicit constant is independent of $P$, $\eta$ and $\bal$. Here $\| \cdot \|$ denotes the distance to the closest integer.
Let us denote
$$
\sigma = \max_{\c \in \ZZ^R \setminus \{\mathbf{0} \} } \mathcal{H}_{\c.\mathbf{F}}.
$$
\begin{lem}
\label{lem2.5}
Let $P > 0$, $\bal \in \mathbb{R}^R$ and  $\eta \in (0,1)$.
Then one of the following alternatives holds.
\begin{enumerate}[label= $(\textnormal{\roman*})$]
\item For any $\varepsilon > 0$, we have
$$
|S(\bal) | \ll {P}^{n -  \frac{(n - \sigma) \eta}{2^{d-1}}  + \varepsilon},
$$
where the implicit constant depends only on $d$, $n$ and $\varepsilon$.

\item There exist $q, a_1, \ldots, a_R \in \ZZ$ such that $\gcd(q, \a) = 1$,
$$
1 \leq  q  \ll {P}^{R(d-1) \eta}  \quad  \textnormal{and}  \quad
|q \alpha_\ell - a_\ell| \ll {P}^{- d  + R(d-1) \eta}  ~   (1 \leq \ell \leq R),
$$
where the implicit constant depends only on $d$, $n$, $R$ and the coefficients of $\mathbf{F}$.
\end{enumerate}

\end{lem}

\begin{proof}
For each $1 \leq j \leq n$, we let $M_j$ be the matrix whose columns are
$$
\begin{bmatrix}
  \Gamma_{1}(\x_{1}, \ldots, \x_{d-1}, \mathbf{e}_j) \\
  \vdots  \\
  \Gamma_{R}(\x_{1}, \ldots, \x_{d-1}, \mathbf{e}_j)
\end{bmatrix}
$$
for each $(\x_{1}, \ldots, \x_{d-1})$ in  $N^{(d-1)}(P^{\eta}; \bal)$.
We then put these matrices together and define
$$
M = [M_1 \cdots M_n].
$$
We consider two cases depending on the rank of $M$; the case $\textnormal{rank } M = R$ corresponds to alternative (ii) and the case $\textnormal{rank } M < R$ to (i) in the statement of the lemma.

First we suppose $\textnormal{rank } M = R$. Then there exists an $R \times R$ submatrix
$$
M_0
=
[\Gamma_{\ell}(\x^{(s)}_{1}, \ldots, \x^{(s)}_{d-1}, \mathbf{e}_{j_{s}})]_{1 \leq \ell, s \leq R}
$$
whose rank is $R$. By the definition of $N^{(d-1)}({P}^{\eta}; \bal)$, %we have
%$$
%\left \| \sum_{i = 1}^R \alpha_i \Gamma_{i}(\x^{(\ell)}_{1}, \ldots, \x^{(\ell)}_{d-1},  \mathbf{e}_{j_{\ell}}) \right\| < \hat{P}^{- d + (d-1) \eta }.
%$$
%In particular,
there exist  $m_{s} \in \ZZ$ and $|\gamma_{s}| < {P}^{- d + (d-1) \eta }$ satisfying
$$
\sum_{\ell = 1}^R \alpha_\ell \Gamma_{\ell}(\x_{1}^{(s)}, \ldots, \x_{d-1}^{(s)}, \mathbf{e}_{j_{s}})
= m_{s} + \gamma_{s},
$$
for each $1 \leq s \leq R$.
Let $\mathfrak{C} =  [ \mathfrak{c}_{\ell, s}]_{1 \leq \ell, s \leq R}$ be the cofactor matrix of $M_0$, i.e. it is the matrix with entries in $\ZZ$ such that
$$
M_0 \mathfrak{C}^T =  (\det M_0)  I_{R \times R} = \mathfrak{C} M_0^T,
$$
where $I_{R \times R}$ is the $R \times R$ identity matrix.
Therefore, we obtain
\begin{eqnarray}
\notag
\begin{bmatrix}
(\det M_0) \alpha_1 - \sum_{s = 1}^R  \mathfrak{c}_{1, s} m_{s}
\\
\vdots
\\
(\det M_0) \alpha_R - \sum_{s = 1}^R  \mathfrak{c}_{R, s} m_{s}
\end{bmatrix}
=
\mathfrak{C} M_0^T \begin{bmatrix}
\alpha_1
\\
\vdots
\\
\alpha_R
\end{bmatrix}
-
\mathfrak{C}  \begin{bmatrix}
m_1
\\
\vdots
\\
m_R
\end{bmatrix}
=
\mathfrak{C}  \begin{bmatrix}
\gamma_1
\\
\vdots
\\
\gamma_R
\end{bmatrix},
\end{eqnarray}
and it follows that
$$
\left| (\det M_0) \alpha_\ell - \sum_{s = 1}^R  \mathfrak{c}_{\ell, s} m_{s} \right|
\ll {P}^{\eta (R-1) (d-1)} {P}^{- d + (d-1) \eta }
=  {P}^{- d + R (d-1) \eta } ~(1 \leq \ell \leq R).
$$
Then setting
$$
q = \frac{ |\det M_0| }{\mathfrak{d} }  \quad   \textnormal{and}
 \quad  a_\ell = \frac{ \textnormal{sign} ( \det M_0 ) }{\mathfrak{d} } \sum_{s = 1}^R  \mathfrak{c}_{\ell, s} m_{s}  ~  (1 \leq \ell \leq R),
$$
where
$$
\mathfrak{d} = \gcd \left(  \det M_0,  \sum_{s = 1}^R  \mathfrak{c}_{1, s} m_{s}, \ldots, \sum_{s = 1}^R  \mathfrak{c}_{R, s} m_{s} \right),
$$
we establish alternative (ii) of the lemma on noting that
$$
1 \leq |\det M_0| \ll {P}^{\eta R (d-1)}.
$$

Next we suppose $\textnormal{rank } M < R$.
Then there exists $\c \in \ZZ^R \setminus \{ \mathbf{0} \}$ such that
$$
0 = \sum_{\ell = 1}^R c_\ell \Gamma_{\ell}(\x_{1}, \ldots, \x_{d-1}, \mathbf{e}_{j}) ~(1 \leq j \leq n),
$$
for all $(\x_{1}, \ldots, \x_{d-1})$ in  $N^{(d-1)}({P}^{\eta}; \bal)$.
Since
$$
\sum_{\ell = 1}^R c_\ell \Gamma_{\ell}(\x_{1}, \ldots, \x_{d-1}, \mathbf{e}_{j}) = \Gamma_{\c.\mathbf{F}}(\x_{1}, \ldots, \x_{d-1}, \mathbf{e}_{j}) ~(1 \leq j \leq n),
$$
it then follows that
\begin{eqnarray}
\notag
&& \#  N^{(d-1)}({P}^{\eta}; \bal)
\\
\notag
&\leq& \# \{ \x_1, \ldots, \x_{d-1} \in ( [-P^{\eta}, P^{\eta}] \cap \ZZ )^n:  \Gamma_{\c.\mathbf{F}}(\x_{1}, \ldots, \x_{d-1}, \mathbf{e}_{j}) = 0 ~(1 \leq j \leq n) \}
\\
\notag
&\ll&
P^{\eta((d-2)n + \mathcal{H}_{\c.\mathbf{F}} )},
\end{eqnarray}
where the final inequality is obtained by Lemma \ref{main lem} with $G = \c.\mathbf{F}$ and $B = P^{\eta}$.
Thus we obtain from (\ref{SSeqn}) that
\begin{eqnarray}
|S(\bal)|^{2^{d-1}}
\ll
P^{n( 2^{d-1}-d+1) + \varepsilon} P^{ (1 - \eta ) (d-1) n  } P^{ \eta(  (d-2) n + \sigma ) }
= P^{n 2^{d-1} + \varepsilon} P^{ - \eta ( n - \sigma ) },
%\\
%\notag
%&\ll& P^{n 2^{d-1} + 2 \varepsilon} P^{ - \eta ( n - \sigma ) }.
\end{eqnarray}
and this estimate gives alternative (i) in the statement of the lemma.
\end{proof}

This lemma shows that $K$ in \cite[Lemma 4.3]{Bir}, which is defined to be
$$
K = \frac{n - \dim V_{\mathbf{F}}^*}{2^{d-1}}
$$
may be replaced by
$$
K = \frac{n - \sigma}{2^{d-1}}.
$$
Finally, by following the proof in \cite{Bir} with Lemma \ref{lem2.5}, we recover Theorem \ref{Birchthm} with $\dim V_{\mathbf{F}}^*$ replaced by
$\sigma$, namely Theorem \ref{mainthm}.

\section{Proof of Lemma \ref{main lem}}
\label{proof}
For each $0 \leq r \leq n$, we define
$$
Z(r) =
\{  \x_1, \ldots, \x_{d-2} \in \mathbb{A}_{\mathbb{C}}^n: \textnormal{rank } [\Gamma_{G} (\x_1, \ldots, \x_{d-2}, \mathbf{e}_a, \mathbf{e}_b )]_{ \substack{1 \leq a, b \leq n} }  =  r \}
$$
and
$$
Z_+(r) =
\{  \x_1, \ldots, \x_{d-2} \in \mathbb{A}_{\mathbb{C}}^n: \textnormal{rank } [\Gamma_{G} (\x_1, \ldots, \x_{d-2}, \mathbf{e}_a, \mathbf{e}_b )]_{ \substack{1 \leq a, b \leq n} }  \leq   r \}.
$$
We remark that $Z_+(r)$ is an affine variety defined by homogeneous forms. In order to understand the dimension of $Z_+(r)$, we make use of the following lemma.
\begin{lem}
\label{seconf der}
Let $d \geq 2$. For any $1 \leq a, b \leq n$, we have
$$
\Gamma_{G} (\x, \ldots, \x, \mathbf{e}_a, \mathbf{e}_b ) = \frac{d! }{d (d-1)} \cdot \frac{\partial^2 G}{ \partial x_a \partial x_b } (\x).
$$
\end{lem}
\begin{proof}
Let us fix a choice of $1 \leq \ell_1 <  \cdots <  \ell_{m} \leq n$ for some $1 \leq m \leq d$.
By linearity it suffices to prove the statement for
$$
G (x_1, \ldots, x_n) = x_{\ell_1}^{c_1} \cdots x_{\ell_{m}}^{c_m} = \sum_{\mathbf{j}} G_{\mathbf{j}} \tsp x_{j_1} \cdots x_{j_d},
$$
where $c_1 + \cdots + c_m = d$ and
$$
G_{\mathbf{j}} =
\begin{cases}
\frac{c_1 ! \cdots c_m !  }{d!} & \mbox{if }  x_{j_1} \cdots x_{j_d} = x_{\ell_1}^{c_1} \cdots x_{\ell_{m}}^{c_m}, \\
0 & \mbox{otherwise}.
\end{cases}
$$
First we deal with the case $a \neq b$. We suppose that both $a$ and $b$ occur in $\{\ell_1, \ldots, \ell_m \}$, as the result is trivial otherwise.
Let $\ell_h = a$ and $\ell_k = b$. By the definition of $\Gamma_G$, it follows that
\begin{eqnarray}
\notag
\Gamma_G(\x, \ldots, \x, \mathbf{e}_a, \mathbf{e}_b )
&=&
d! \sum_{\mathbf{j}} G_{\mathbf{j}} x_{j_1} \cdots x_{j_{d-2}} \mathbbm{1}_{j_{d-1} = a } \mathbbm{1}_{j_{d} = b }
\\
\notag
&=&
d! \frac{x_{\ell_1}^{c_1} \cdots x_{\ell_{m}}^{c_m}}{x_a x_b} \sum_{\mathbf{j}} G_{\mathbf{j}} \mathbbm{1}_{j_{d-1} = a } \mathbbm{1}_{j_{d} = b }
\notag
\\
\notag
&=&
\frac{d!}{c_h c_k}  \cdot
\frac{\partial^2 G}{\partial x_a \partial x_b} (\x) \sum_{\mathbf{j}} G_{\mathbf{j}} \mathbbm{1}_{j_{d-1} = a } \mathbbm{1}_{j_{d} = b }
\end{eqnarray}
and
$$
\sum_{\mathbf{j}} G_{\mathbf{j}} \mathbbm{1}_{j_{d-1} = a } \mathbbm{1}_{j_{d} = b }
=
\frac{c_1 ! \cdots c_m !  }{d!}  \cdot \frac{(d-2)! c_h c_k  }{ c_1 ! \cdots c_m ! } = \frac{c_h c_k}{ d (d-1) }.
$$
Therefore, we obtain
$$
\Gamma_{G}(\x, \ldots, \x, \mathbf{e}_a, \mathbf{e}_b ) = \frac{d!}{d (d-1)}
\cdot \frac{\partial^2 G}{\partial x_a \partial x_b} (\x).
$$

Next let us suppose $a = b$. The proof is similar to the above case, but we include it for completeness.
We use the same notation as above. Let us assume $c_h > 1$, as the result is trivial otherwise.
Then we have
\begin{eqnarray}
\notag
\Gamma_G(\x, \ldots, \x, \mathbf{e}_a, \mathbf{e}_a )
&=&
d! \sum_{\mathbf{j}} G_{\mathbf{j}} x_{j_1} \cdots x_{j_{d-2}} \mathbbm{1}_{j_{d-1} = a } \mathbbm{1}_{j_{d} = a }
\\
\notag
&=&
d! \frac{x_{\ell_1}^{c_1} \cdots x_{\ell_{m}}^{c_m}}{x_a^2} \sum_{\mathbf{j}} G_{\mathbf{j}} \mathbbm{1}_{j_{d-1} = a } \mathbbm{1}_{j_{d} = a }
\notag
\\
\notag
&=&
\frac{d!}{c_h (c_h - 1)}  \cdot
\frac{\partial^2 G}{\partial x_a^2} (\x) \sum_{\mathbf{j}} G_{\mathbf{j}} \mathbbm{1}_{j_{d-1} = a } \mathbbm{1}_{j_{d} = a }
\end{eqnarray}
and
$$
\sum_{\mathbf{j}} G_{\mathbf{j}} \mathbbm{1}_{j_{d-1} = a } \mathbbm{1}_{j_{d} = a }
=
\frac{c_1 ! \cdots c_m !  }{d!}  \cdot \frac{(d-2)! c_h (c_h - 1)  }{ c_1 ! \cdots c_m ! } = \frac{c_h (c_h - 1)}{ d (d-1) }.
$$
Therefore, we obtain
$$
\Gamma_G(\x, \ldots, \x, \mathbf{e}_a, \mathbf{e}_a ) = \frac{d!}{d (d-1)}
\cdot \frac{\partial^2 G}{\partial x_a^2} (\x).
$$
\end{proof}

We need the following estimate, which is \cite[(3.1)]{Bro}.
\begin{lem}
\label{lin alg}
Let $T \subseteq \mathbb{A}^{\nu}_{\mathbb{C}}$ be an affine variety (not necessarily irreducible). Then
$$
\# \{ \t \in T \cap \ZZ^{\nu} :  |\t| \leq B  \}
\ll B^{\dim T},
$$
where the implicit constant depends only on $\nu$ and $\deg T$.
\end{lem}

With these two lemmata in hand, the case $d=2$ of Lemma \ref{main lem} follows easily:
\begin{eqnarray}
%\label{multi lin eqn'}
\notag
&&\# \{ \x_1 \in ([-B, B] \cap \ZZ )^n:  \Gamma_G (\x_1, \mathbf{e}_b )  = 0  ~  (1 \leq b \leq n)   \}
\\
\notag
&=& \# \{ \y \in ([-B, B] \cap \ZZ )^n:  y_1 \Gamma_G (\mathbf{e}_1 , \mathbf{e}_b )  + \cdots +  y_n  \Gamma_G (\mathbf{e}_n , \mathbf{e}_b )  = 0  ~  (1 \leq b \leq n)   \}
\\
\notag
&=& \# \{  \mathbf{y}  \in ([-B,B] \cap \ZZ)^n :  \mathbf{y}^T \cdot  [\Gamma_G (\mathbf{e}_a, \mathbf{e}_b )]_{ \substack{1 \leq a, b \leq n}}   = \mathbf{0}
 \}
\\
&=& \# \{  \mathbf{y}  \in ([-B,B] \cap \ZZ)^n :   H_G^T \cdot \mathbf{y} = \mathbf{0} \}
\notag
\\
\notag
&\ll& B^{n - \textnormal{rank }  H_{G}}
\\
\notag
&=& B^{(d-2)n + \mathcal{H}_{G}}.
\end{eqnarray}
Here the implicit constant depends only on $n$. %(since we know that $0 \leq \textnormal{rank } H_{G} \leq n$).
Therefore, we assume $d > 2$ for the remainder of this section.
By lemma \ref{seconf der} it follows that
\begin{eqnarray}
\notag
&& Z_+(r) \tsp \bigcap \bigcap_{ 2 \leq k \leq d-2 }   \{  \x_1, \ldots, \x_{d-2} \in \mathbb{A}_{\mathbb{C}}^n:  \x_{k} =  \x_{1}  \}
\\
\notag
&=& \{  \x_1 \in \mathbb{A}_{\mathbb{C}}^n:   \textnormal{rank } [\Gamma_{G} (\x_1, \ldots, \x_{1}, \mathbf{e}_a, \mathbf{e}_b )]_{ \substack{1 \leq a, b \leq n} }  \leq  r  \}
\\
\notag
&=&
\{  \x \in \mathbb{A}_{\mathbb{C}}^n:   \textnormal{rank }  H_{G}(\x) \leq  r \}.
\end{eqnarray}
Therefore, we obtain
\begin{equation}
\notag
\dim \{  \x \in \mathbb{A}_{\mathbb{C}}^n:   \textnormal{rank } H_{G}(\x) \leq  r \}
\geq
\dim Z_+(r) + n  - (d-2)n,
\end{equation}
i.e.
\begin{eqnarray}
\label{dim}
\dim Z_+(r)  &\leq& (d-3)n +  \dim \{  \x \in \mathbb{A}_{\mathbb{C}}^n:   \textnormal{rank } H_{G}(\x) \leq  r  \}
\\
\notag
&\leq&
(d-3)n +  r + \mathcal{H}_{G},
\end{eqnarray}
where the second inequality follows from the definition of $\mathcal{H}_{G}$ given in (\ref{def}).

We are now in position to complete the proof of Lemma \ref{main lem}. First, by the multilinearity of $\Gamma_G$ and Lemma \ref{lin alg}, we observe that
\begin{eqnarray}
\label{multi lin eqn'}
&&\# \{ \x_1, \ldots, \x_{d-1} \in ([-B, B] \cap \ZZ )^n:  \Gamma_{G} (\x_1, \ldots, \x_{d-1}, \mathbf{e}_i )  = 0  ~  (1 \leq i \leq n)   \}
\\
\notag
&=&
\sum_{  \x_1, \ldots, \x_{d-2} \in ([-B, B] \cap \ZZ)^n }
\\
\notag
&& \# \{  \mathbf{y}  \in ([-B,B] \cap \ZZ)^n :   \mathbf{y}^T \cdot   [\Gamma_{G} (\x_1, \ldots, \x_{d-2}, \mathbf{e}_a, \mathbf{e}_b )]_{ \substack{1 \leq a, b \leq n}}  = \mathbf{0}
 \}
\\
&=&
\sum_{0 \leq r \leq n} \  \sum_{  (\x_1, \ldots, \x_{d-2}) \in Z(r) \cap ([-B,B] \cap \ZZ)^{(d-2)n} }
\notag
\\
&& \# \{  \mathbf{y}  \in ([-B,B] \cap \ZZ)^n :    \mathbf{y}^T \cdot   [\Gamma_{G} (\x_1, \ldots, \x_{d-2}, \mathbf{e}_a, \mathbf{e}_b )]_{ \substack{1 \leq a, b \leq n}}  = \mathbf{0}
\}
\notag
\\
\notag
&\ll&  \sum_{0 \leq r \leq n} \  \sum_{  (\x_1, \ldots, \x_{d-2}) \in Z(r) \cap  ([-B,B] \cap \ZZ)^{(d-2)n} } B^{n-r},
\end{eqnarray}
where the implicit constant depends only on $n$.
Next, by Lemma \ref{lin alg} and (\ref{dim}), we have
\begin{eqnarray}
\notag
\#  Z(r) \cap  ([-B,B] \cap \ZZ)^{(d-2)n}
\leq
\#  Z_+(r) \cap  ([-B,B] \cap \ZZ)^{(d-2)n}
\ll
B^{\dim Z_+(r)} \leq B^{(d-3)n +  r + \mathcal{H}_G},
\end{eqnarray}
where the implicit constant depends only on $d, n$ and $r$.
By substituting this estimate into (\ref{multi lin eqn'}), we obtain
\begin{eqnarray}
\notag
&&\# \{ \x_1, \ldots, \x_{d-1} \in ([-B, B] \cap \ZZ )^n:  \Gamma_{G} (\x_1, \ldots, \x_{d-1}, \mathbf{e}_i )  = 0  ~  (1 \leq i \leq n)   \}
\\
\notag
&\ll&  \sum_{0 \leq r \leq n}  B^{(d-3)n +  r + \mathcal{H}_{G}}  B^{n-r}
\\
&=&
\notag
n B^{(d-2)n + \mathcal{H}_{G}},
\end{eqnarray}
as desired.

\section{Examples}
\label{EXAMP}
\begin{example}
\label{example1}
Let $n = 2m + k$, $d > 1$ and
$$
F(\x) =   \sum_{j = 1}^m a_j x_{2j - 1}^{d} x_{2j}^{d} + \sum_{i = 1}^{k} b_i x_{2m + i}^{2d},
$$
where $a_j, b_i \in \ZZ \setminus \{ 0 \}$.
Then
\begin{eqnarray}
V_F^* =
\bigcap_{1 \leq j \leq m}
\{ \x \in \mathbb{A}^{n}_{\CC}: x_{2j - 1} x_{2j} = 0 \} \  \bigcap \tsp   \bigcap_{1 \leq i \leq k} \{ \x \in \mathbb{A}^{n}_{\CC}: x_{2m + i} = 0 \}
\notag
\end{eqnarray}
and
$$
\dim V_F^* = m.
$$
On the other hand, $H_F$ is a block diagonal matrix
$$
H_F(\x) =
\begin{bmatrix}
 a_1 d B_1 & \cdots  & [0]   & [0] \\
  \vdots  & \ddots  & \vdots &  \vdots  \\
  [0] & \cdots  &  a_m d B_m  & [0] \\
  [0] &  \cdots & [0] &  2 d (2d-1) D
\end{bmatrix},
$$
where
$$
B_j = \begin{bmatrix}
  (d-1)  x_{2j - 1}^{d - 2}  x_{2j}^{d} & d x_{2j-1}^{d-1}  x_{2j}^{d-1}   \\
  d x_{2j-1}^{d-1}  x_{2j}^{d-1} & (d-1)  x_{2j - 1}^{d}  x_{2j}^{d - 2}
\end{bmatrix}
 ~
\textnormal{ and }
 ~
D = \begin{bmatrix}
   b_1 x_{2m + 1}^{2 d-2}  & \cdots  & 0  \\
  \vdots  & \ddots  & \vdots   \\
  0 & \cdots  &   b_k x_{2m + k}^{2 d -2}
\end{bmatrix}.
$$
It is clear that
$$
\dim \{ \x \in \mathbb{A}^n_{\CC}:  \textnormal{rank }  H_{F}(\x)  \leq n  \} \leq n
$$
and
$$
\dim \{ \x \in \mathbb{A}^n_{\CC}:  \textnormal{rank }  H_{F}(\x)  \leq n-1  \} \leq n-1.
$$
For each $0 \leq  r  \leq n - 2$, we have
$$
\{ \x \in \mathbb{A}^n_{\CC}:  \textnormal{rank }  H_{F}(\x)  \leq r  \} \subseteq \bigcup_{ 1 \leq i_1 < \cdots < i_{n - r} \leq n } \{  \x \in \mathbb{A}^n_{\CC}:    x_{i_1} = \cdots =  x_{i_{n -r} } = 0   \}
$$
and
$$
\dim \{ \x \in \mathbb{A}^n_{\CC}:  \textnormal{rank }  H_{F}(\x)  \leq r  \} \leq r.
$$
Therefore, we obtain $\mathcal{H}_F = 0$ in this case.
\end{example}

\begin{example}
\label{example2}
Let $n = 2m$, $d > 3$ and
$$
F(\x) =   \sum_{j = 1}^m a_j \left( x_{2j - 1} x_{2j}^{d - 1} -  \frac{1}{d} x_{2j - 1}^d \right),
$$
where $a_j \in \ZZ \setminus \{ 0 \}$.
Then
\begin{eqnarray}
V_F^* =
\bigcap_{1 \leq j \leq m}
\left(
\{ \x \in \mathbb{A}^{n}_{\CC}: x^{d-1}_{2j} -  x_{2j - 1}^{d-1} = 0 \}
\tsp  \bigcup \tsp
\{ \x \in \mathbb{A}^{n}_{\CC}: x_{2j - 1} x_{2j} = 0 \}
\right)
\notag
\end{eqnarray}
and
$$
\dim V_F^* = m.
$$
On the other hand, $H_F$ is a block diagonal matrix
$$
H_F(\x) =
\begin{bmatrix}
 a_1 (d-1) B_1 & \cdots  & [0]   \\
  \vdots  & \ddots  & \vdots   \\
  [0] & \cdots  &  a_m (d-1) B_m
\end{bmatrix},
$$
where
$$
B_j = \begin{bmatrix}
   - x_{2j - 1}^{d-2}   &  x_{2j}^{d - 2}  \\
    x_{2j}^{d - 2} &  (d-2) x_{2j - 1} x_{2j}^{d - 3}
\end{bmatrix}.
$$
Since
$$
\textnormal{rank } B_j =
\begin{cases}
  1 & \mbox{if precisely one of }  (d-2) x_{2j-1}^{d-1} + x_{2j}^{d-1} = 0 \mbox{ or } x_{2j} = 0 \mbox{ holds}, \\
  0 & \mbox{if } x_{2j -1} = x_{2j} = 0,
\end{cases}
$$
it can be verified that
$$
\dim \{ \x \in \mathbb{A}^n_{\CC}:  \textnormal{rank }  H_{F}(\x)  \leq r  \} \leq r,
$$
for each $0 \leq  r  \leq n$.
Therefore, we obtain $\mathcal{H}_F = 0$ in this case as well.

\iffalse
It is clear that
$$
\dim \{ \x \in \mathbb{A}^n_{\CC}:  \textnormal{rank }  H_{F}(\x)  \leq n  \} \leq n
$$
and
$$
\dim \{ \x \in \mathbb{A}^n_{\CC}:  \textnormal{rank }  H_{F}(\x)  \leq n-1  \} \leq n-1.
$$

For each $0 \leq  r  \leq n - 2$, we also have
%$$
%\{ \x \in \mathbb{A}^n_{\CC}:  \textnormal{rank }  H_{F}(\x)  \leq r  \} \subseteq \bigcap_{ 1 \leq i_1 < \cdots < i_{r} \leq n } \{  \x \in \mathbb{A}^n_{\CC}:    x_{i_1} \cdots x_{i_{r} } = 0   \}
%$$
%and
$$
\dim \{ \x \in \mathbb{A}^n_{\CC}:  \textnormal{rank }  H_{F}(\x)  \leq r  \} \leq r.
$$
Therefore, we obtain $\mathcal{H}_F = 0$ in this case.
\fi
\end{example}

\end{document}